\documentclass[12pt,reqno]{amsart}
\usepackage{amsfonts,amssymb,latexsym,amsmath, amsxtra, mathrsfs}
\usepackage{verbatim}
\usepackage{diagbox}

\theoremstyle{plain}
\newtheorem{theorem}{Theorem}[section]
\newtheorem{corollary}[theorem]{Corollary}

\newtheorem{lemma}[theorem]{Lemma}
\newtheorem{proposition}[theorem]{Proposition}

\newtheorem*{theorem*}{Theorem}
\theoremstyle{definition}

\theoremstyle{remark}
\newtheorem*{remark}{Remark}

\numberwithin{equation}{section}

\newcommand{\N}{\mathbb N}
\newcommand{\Z}{\mathbb Z}

\def\({\left(}
\def\){\right)}
\def\<{\left<}
\def\>{\right>}

\newcommand{\Var}{\operatorname{Var}}

\renewcommand{\pmod}[1]{\  \,  \left(  \mathrm{mod} \,  #1 \right)}

\begin{document}
\allowdisplaybreaks
	
\begin{abstract}
In this article we extend a theorem of Andrews, Crippa, and Simon on
the asymptotic behavior of polynomials defined by a general class of
recursive equations. Here the polynomials are in the variable $q$,
and the recursive definition at step $n$ introduces a polynomial in
$n$. Our extension replaces the polynomial in $n$ with either an 
exponential or periodic function of $n$.
\end{abstract}

\title[$q$-series identities extending work of Andrews, Crippa, and Simon]
{Some $q$-series identities extending work of Andrews, Crippa, and Simon
on sums of divisors functions}

\thanks{The research of the first author is supported by the Alfried Krupp Prize for 
Young University Teachers of the Krupp foundation, and the research leading to these 
results receives funding from the European Research Council under the European Union's 
Seventh Framework Programme (FP/2007-2013) / ERC Grant agreement n. 335220 - AQSER}

\author[K. Bringmann]{Kathrin Bringmann}
\author[C. Jennings-Shaffer]{Chris  Jennings-Shaffer}
	
\address{Kathrin Bringmann, University of Cologne, Department of Mathematics and Computer Science, Weyertal 86-90, 50931 Cologne, Germany}
\email{kbringma@math.uni-koeln.de}
\address{Chris  Jennings-Shaffer, University of Cologne, Department of Mathematics and Computer Science, Weyertal 86-90, 50931 Cologne, Germany}
\email{cjenning@math.uni-koeln.de}
\maketitle

\section{Introduction and statement of results}
In 1993, Collenberg, Crippa, and Simon  \cite{CollenbergCrippaSimon1} found that the expected
value and variance of a certain random variable on acyclic digraphs can be
expressed in terms of the generating functions of the number of divisors
of an integer and the sum of divisors of an integer, respectively. Specifically,
the probability space is the collection of acyclic digraphs with vertices 
$1$, $2$, $\dotsc$, $k$ (for fixed $k$) and the probability that a
directed edge exists between any two vertices $i$ and $j$,
with $1\le i<j\le k$, is uniformly $1-q$ (for fixed $q$ with $0<q<1$). For each fixed $k$,
the random variable $\gamma_k^*$ is defined as the number of vertices reachable from
the vertex $1$. Among their results are the identities
\begin{gather*}
\lim_{k\rightarrow\infty} (k - E(\gamma_k^*))
=
	\sum_{n\ge1}  \sum_{d\mid n} q^n
,\qquad\qquad\qquad
\lim_{k\rightarrow\infty} \Var(\gamma_k^*)
=
	\sum_{n\ge1}  \sum_{d\mid n} dq^n
.
\end{gather*}

Shortly after, Andrews, Crippa, and Simon \cite{AndrewsCrippaSimon1} found
that the relevant proofs, which are based on the limiting behavior of certain
recursively defined polynomials, could be recast and directly handled with
$q$-series techniques. Along with a number of identities for $q$-hypergeometric
series related to various sums of powers of divisors functions, their work
also gave the limiting behavior of a large family of polynomials. To state their result, we use the standard notation, for $n\in\N_0\cup\{\infty\}$,
\begin{equation*}
(a;q)_n := \prod_{j=0}^{n-1} \left(1-aq^j\right).
\end{equation*}
and for $k\in\N$ we let $S_k(q)$ denote the generating functions for the sum of powers of divisors functions
\begin{gather*}
S_k(q) 
:= 
\sum_{n\ge1} \sigma_k(n)q^n
,\qquad\qquad
\textnormal{where }\sigma_{k}(n) 
:=
\sum_{d\mid n} d^k
.
\end{gather*}
The following is a slightly reworded statement of Theorem 3.1 of \cite{AndrewsCrippaSimon1}.

\begin{theorem}\label{3.1}
Suppose that $f(n):=\sum_{k\ge0}c_kn^k$ is a polynomial in $n$.
Let $a_n(q)$ be the polynomials in $q$ defined recursively by
\begin{align*}
a_n(q) &= f(n) + \left(1-q^{n-1}\right)a_{n-1}(q) \qquad \mbox{for }n\in\N,\qquad a_0(q) = 0.
\end{align*}
Then
\begin{align*}
\lim_{n\rightarrow\infty}\left( \sum_{1\le\ell\le n} f(\ell) - a_n(q) \right)
&=
\sum_{j\ge1} h_j \sum_{n\ge1} \frac{(-1)^{n-1}q^{\frac{n(n+1)}{2}}}{(1-q^n)^j(q;q)_n }
,
\end{align*}
with
\begin{gather*}
h_1 
= 
	c_0
,\qquad\qquad
h_j
=
	\sum_{\ell\ge j-1} (-1)^{\ell-j-1}\binom{\ell-1}{j-2}\ell!
	\sum_{k\ge\ell} c_k S(k,\ell),
\end{gather*}
where $S(k,\ell)$ are the Stirling numbers of the second kind.
In particular, if the $c_k$ are rational, then so are the $h_j$.
Furthermore, each 
\begin{gather*}
 \sum_{n\ge1} \frac{(-1)^{n-1}q^{\frac{n(n+1)}{2}}}{(1-q^n)^j(q;q)_n }
\end{gather*}
can be written as polynomial in $S_0(q),S_1(q),\dots,S_j(q)$ with
rational coefficients.
\end{theorem}

At the end of their article, Andrews, Crippa, and Simon posed the question
of determining a similar result when $f(n)$ is replaced by a periodic
function, and gave the following identity (without proof).  
If $f(n) = (-1)^n$, then
\begin{gather*}
\lim_{n\rightarrow\infty}\left( \sum_{1\le\ell\le n} f(\ell) - a_n(q) \right)
=
\sum_{j\ge1} (-1)^j q^{j^2}.
\end{gather*}
We give two extensions of Theorem \ref{3.1}. The first is when $f(n)=b^n$ and the
second is when $f(n)$ is periodic. Of course, the two overlap when $b$ is a root
of unity, such as in the example above. Our first theorem  for $f(n)=b^n$ is as follows.

\begin{theorem}\label{Theorem:GeneralizationOne}
Suppose that $b\in\mathbb{C}\setminus\{1\}$, and $a_n(q)$ is the sequence of polynomials in $q$ defined recursively by
\begin{align*}
a_n(q) &= b^n + \left(1-q^{n-1}\right)a_{n-1}(q) \qquad \mbox{for }n\in\N,\qquad a_0(q) = 0.
\end{align*}
Then, for $|q|<\min(|b|^{-1},1)$, we have
\begin{gather*}
\lim_{n\rightarrow\infty}\left( \sum_{1\le\ell\le n} b^\ell - a_n(q)\right)
=\lim_{n\rightarrow\infty}\left( \frac{b-b^{n+1}}{1-b} - a_n(q) \right)
=
\frac{b}{1-b} - \frac{b(q;q)_\infty}{(b;q)_\infty}.
\end{gather*}
\end{theorem}
\begin{remark}
We note that the case $b=1$ of Theorem \ref{Theorem:GeneralizationOne} is
covered by Theorem \ref{3.1}. In particular, when $b=1$,
\begin{gather*}
\lim_{n\rightarrow\infty} \left(n-a_n(q)\right) = S_0(q),
\end{gather*}
and in fact Collenberg, Crippa, and Simon 
\cite[equation (8)]{CollenbergCrippaSimon1} observed that
$a_k(q) = E(\gamma_k^*)$. 
\end{remark}

The following theorem gives the extension for when $f(n)$ is periodic.

\begin{theorem}\label{Theorem:GeneralizationTwo}
Suppose $f(n)$ is a periodic sequence with period $N$
and $a_n(q)$ is the sequence of polynomials in $q$ defined recursively by
\begin{align*}
a_n(q) &= f(n) +\left(1-q^{n-1}\right)a_{n-1}(q) \qquad \mbox{for }n\in\N,\qquad a_0(q) = 0.
\end{align*}
Setting
\begin{gather*}
c_k := \frac{1}{N}\sum_{1\le j\le N} f(j)\zeta_N^{(1-j)k},
\end{gather*}
we obtain for $|q|<1$,
\begin{gather*}
\lim_{n\rightarrow\infty}\left( \sum_{1\le\ell\le n} f(\ell) - a_n(q)\right)
=
	c_0S_0(q)
	-
	(q;q)_\infty
	\sum_{1\le k\le N-1} 
	\frac{c_k}{(\zeta_N^k;q)_\infty}
	+
	\sum_{1\le k\le N-1} \frac{c_k}{1-\zeta_N^k}
.
\end{gather*}
\end{theorem}

Since Theorem \ref{Theorem:GeneralizationTwo} does not explicitly demonstrate
that the coefficients of the resulting series are elements of
$f(1)\mathbb{Z}+f(2)\mathbb{Z}+\dotsb+f(N)\mathbb{Z}$, we give the following corollary.
\begin{corollary}\label{cor:recur}
Suppose that $f(n)$ is a periodic sequence with period $N$
and $a_n(q)$ is the sequence of polynomials in $q$ defined recursively by
\begin{align*}
a_n(q) &= f(n) + \left(1-q^{n-1}\right)a_{n-1}(q) \qquad \mbox{for }n\in\N,\qquad a_0(q) = 0.
\end{align*}
Then for $|q|<1$, 
\begin{gather*}
\lim_{n\rightarrow\infty}\left( \sum_{1\le\ell\le n} f(\ell) - a_n(q) \right)
=
	(q;q)_\infty
	\sum_{n\ge0}\frac{q^n}{(q;q)_n}
	\sum_{1\le j\le N} f(j) \left\lceil\frac{n+1-j}{N}\right\rceil
.
\end{gather*}
\end{corollary}

The rest of the article is organized as follows. In Section 2, we give the 
additional relevant
notation, definitions, and preliminary identities. In Section 3, we prove
our two generalizations of Theorem \ref{3.1}, which are Theorems 
\ref{Theorem:GeneralizationOne} and \ref{Theorem:GeneralizationTwo}
and Corollary \ref{cor:recur}. 
In Section 4, we compute two examples, which are given as Corollaries \ref{cor:imply} and
\ref{cor:recur2}, one of which is the identity stated above for $f(n)=(-1)^n$.

\section{Preliminaries}

In this section, we recall some basic facts about $q$-series, which are required for this paper.
The following is Euler's identity (see \cite{GasperRahman1}, equation (II. 1)).
\begin{lemma}\label{lem:xqsum}
	We have  for $|x|<1$
	\begin{equation*}
	\sum_{n\geq 0} \frac{x^n}{(q;q)_n} = \frac{1}{(x;q)_\infty}.
	\end{equation*}
\end{lemma}

We require the following representations of $S_0(q)$ as $q$-hypergeometric series (see \cite{Uchimura1}).
\begin{lemma}\label{lem:S0}
	We have
	\begin{align*}
	S_0(q)
	=
	(q;q)_\infty \sum_{n\ge0}\frac{nq^n}{(q;q)_n}
	= \sum_{n\geq 1} \frac{(-1)^{n+1} q^{\frac{n(n+1)}{2}}}{\left(1-q^n\right)(q;q)_n}.
	\end{align*}
	
\end{lemma}
We also use the following limiting case of the $q$-Gauss summation \cite[equation (II.8)]{GasperRahman1}, where
\begin{gather*}
{_2\phi_1}\left(a,b;c;q,x\right)
:=
	\sum_{n\geq 0}\frac{(a; q)_n(b; q)_n}{(q; q)_n (c; q)_n}x^n.
\end{gather*}
\begin{lemma}\label{lem:Cauchy}
	We have
	\begin{gather*}
	\lim_{a\rightarrow\infty}
	{_2\phi_1}\left(a,b;c;q,\frac{c}{ab}\right)
	=\frac{\left(\frac{c}{b};q\right)_\infty}{(c;q)_\infty}
	.
	\end{gather*}
\end{lemma}

For our two examples (Corollaries \ref{cor:imply} and \ref{cor:recur2}), we use 
without mention the following well-known product expansions, e.g. see \cite[Corollary 2.10]{Andrews1},
\begin{gather*}
\frac{(q;q)_\infty}{(-q;q)_\infty}
=
	\sum_{n\in\Z} (-1)^n q^{n^2} 
,\qquad\qquad\qquad
\frac{(q^2;q^2)_\infty}{(q;q^2)_\infty}
=
	\sum_{n\ge1} q^{\frac{n(n+1)}{2}} 
.
\end{gather*}

We also make use of a result that is sometimes referred to as Appell's Comparison Theorem,
which is common when dealing with limiting cases of functional equations and recurrences.
The following statement is a slight extension of Theorem 8.2 in \cite{Rud} to allow for complex coefficients.
\begin{proposition}\label{P:Appell}
	Suppose that $F(x)=\sum_{n\ge0}\alpha_n x^n$ is a power series and $\alpha_\infty := \lim_{n \to \infty} \alpha_n$ exists. Then
	\begin{align*}
	\lim_{x\rightarrow1^-}(1-x) F(x) &= \alpha_\infty.
	\end{align*}
\end{proposition}
We finish this section with an elementary sums of roots of unity identity.

\begin{lemma}\label{lem:rootid}
	We have, for  $j\in\mathbb{Z}$,
	\begin{align*}
	\sum_{1\le k\le N-1}\frac{\zeta_N^{jk}}{1-\zeta_N^k}
	&=
	\frac{N-1}{2} + j - \left\lceil \frac{j}{N} \right\rceil N,
	\end{align*}
	where $\zeta_N:= e^{\frac{2\pi i}{N}}$.
\end{lemma}
\begin{proof}
By periodicity of both sides, we may assume that $-N+1\leq j\leq 0$. We note that
\begin{equation}\label{idx}
\frac{N x^{N+j-1}}{x^N -1} = \sum_{k=1}^N\frac{\zeta_N^{jk}}{x-\zeta_N^k}.
\end{equation}
This identity can be shown by proving that both sides have the same principal parts and both vanish as $x\to\infty$, thus they must then be equal.

Taking the limit $x\to 1$ in \eqref{idx} then gives the claim.
\end{proof}

\section{Proofs of main theorems}

In this section we give the proofs of our theorems and corollaries, beginning with
Theorem \ref{Theorem:GeneralizationOne}.

\begin{proof}[Proof of Theorem \ref{Theorem:GeneralizationOne}]
For $n\ge1$, it is not hard to see that
\begin{align}\label{Eq:ClosedFormForAn}
a_n(q) &= (q;q)_{n-1}\sum_{0\le j\le n-1} \frac{b^{j+1}}{(q;q)_j}.
\end{align}

In order to apply Proposition \ref{P:Appell},
we next prove that
${\displaystyle\lim_{n\rightarrow\infty}}( \frac{b-b^{n+1}}{1-b} - a_n(q))$ exists.
From \eqref{Eq:ClosedFormForAn}, we find that
\begin{align}
\notag
\left| a_n(q) \right| 
&\le 
	\sum_{0\le j\le n-1} |b|^{j+1} \left(-|q|^{j+1};|q|\right)_{n-j-1}
\le 
	(-|q|;|q|)_\infty\sum_{0\le j\le n-1} |b|^{j+1} 
\\
&=
\begin{cases}
	n(-|q|;|q|)_\infty
	&\mbox{if } |b|=1,\\[1ex]
	\displaystyle 
	(-|q|;|q|)_\infty \frac{|b|(1-|b|^n)}{1-|b|}  
	&\mbox{if } |b|\not=1.
\end{cases}
\label{boundan}
\end{align}
This implies that the series
\begin{gather*}
A(x;q) := \sum_{n\ge0} a_n(q)x^n,
\end{gather*}
is absolutely convergent for $|x|<\min(|b|^{-1},1)$. 

We set 
$$
B(x;q) := \sum_{n\in\Z} \beta_n(q) x^n,
$$ 
where
\begin{gather*}
\beta_n(q) := 0 \quad\mbox{for }n\le 0
,\qquad
\beta_n(q) := \frac{b-b^{n+1}}{1-b}-a_n(q) \quad\mbox{for }n\ge1.
\end{gather*}
From the recursion for $a_n(q)$, we have that
\begin{align*}
\beta_n(q) 
&=
	\beta_{n-1}(q)+q^{n-1}a_{n-1}(q)
.	
\end{align*}
This yields that
\begin{align}\label{btail}
\left| \beta_{n+m}(q)-\beta_{n}(q) \right|
&\le
	\sum_{n\le j\le n+m-1} \left| \beta_{j+1}(q)-\beta_{j}(q) \right|
=
	\sum_{n\le j\le n+m-1} \left|a_{j}(q)q^{j}\right|
.
\end{align}
This is the tail of a convergent series since $A$ is absolutely convergent,
and thus the $\beta_n(q)$ form a Cauchy sequence. As such,
${\displaystyle\lim_{n\rightarrow\infty}}(\frac{b-b^{n+1}}{1-b} - a_n(q))$
exists.

To prove the claimed identity in the theorem, we note that 
$\beta_{n}(q) - \beta_{n-1}(q) = q^{n-1}a_{n-1}(q)$ 
is valid for $n=1$. Thus we find that
\begin{align*}
(1-x)B(x;q)
&=
xA(xq;q).
\end{align*}
Proposition \ref{P:Appell} then yields that 
\begin{gather}\label{Eq:UseOfAppell}
\lim_{n\rightarrow\infty}\left( \frac{b-b^{n+1}}{1-b} - a_n(q) \right)
=
\lim_{x\rightarrow1^-}(1-x)B(x;q)
=
A(q;q).
\end{gather}

To finish the proof, we set
\begin{gather*}
F(x) := \sum_{n\ge1} b^n x^n.
\end{gather*}
For $|q|<|b|^{-1}$ and $m\in\N$ we have that, using the geometric series,
\begin{gather}\label{formulaF1}
F(q^{m}) = \frac{bq^{m}}{1-bq^{m}}.
\end{gather}
Using the recurrence for $a_n(q)$, we find that
\begin{gather}\label{Eq:FunctionEqForA}
A(x;q)
=
	\frac{F(x)}{1-x} - \frac{x}{1-x}A(xq;q)
.
\end{gather}
Iterating \eqref{Eq:FunctionEqForA} with $x\mapsto xq$ yields
\begin{align*}
A(x;q)
&=
\sum_{n\ge0} \frac{ (-1)^n F(xq^{n}) q^{ \frac{n(n-1)}{2}} x^{n} }{(x;q)_{n+1}}
.
\end{align*}

In particular, we obtain, using \eqref{formulaF1} and Lemma \ref{lem:Cauchy},
\begin{align*}
A(q;q)
&=	
	-b\sum_{n\ge1} 
	\frac{(-1)^n q^{\frac{n(n+1)}{2}} }{(1-bq^n)(q;q)_n}	
=
	\frac{b}{1-b}
	\left(
		1-
		\lim_{a\rightarrow\infty}
		{_2\phi_1}\left(a,b;bq;q,\frac{q}{a}\right)
	\right)
=	
	\frac{b}{1-b}
	-
	\frac{b(q;q)_\infty}{(b;q)_\infty}.
\end{align*}
With \eqref{Eq:UseOfAppell}, this finishes the proof.
Compared with the proof of Theorem 3.1 from \cite{AndrewsCrippaSimon1},
our proof makes use of the same relation between $A(x;q)$ and $F(x)$. However,
it uses Proposition \ref{P:Appell} and relies on different techniques
to evaluate $A(q;q)$.
\end{proof}

The proof of Theorem \ref{Theorem:GeneralizationTwo} mimics that
of Theorem \ref{Theorem:GeneralizationOne}, up to the calculation of
$A(q;q)$. The details are given below.

\begin{proof}[Proof of Theorem \ref{Theorem:GeneralizationTwo}]\let\qed\relax
Similar to the proof of Theorem \ref{Theorem:GeneralizationOne},
we obtain for $n\geq 1$
\begin{align}\label{Eq:Proof2ClosedFormForAn}
a_n(q) &= (q;q)_{n-1}\sum_{0\le j\le n-1} \frac{f(j+1)}{(q;q)_j}.
\end{align}

Again the proof uses Proposition \ref{P:Appell} applied to the 
series
\begin{gather*}
B(x;q) := \sum_{n\ge0} \beta_n(q) x^n,
\end{gather*}
where
\begin{gather*}
\beta_n(q): = 0 \quad\mbox{for }n\le 0
,\qquad
\beta_n(q) := \sum_{1\le \ell\le n}f(\ell) - a_n(q) \quad\mbox{for }n\ge1.
\end{gather*}
From \eqref{Eq:Proof2ClosedFormForAn}, we find that, exactly as for the proof of \eqref{boundan},
\begin{align*}
\left| a_n(q) \right| 
&\le 
	n(-|q|;|q|)_\infty\max_{1\le j \le N}|f(j)|
.
\end{align*}
As such, the series
\begin{gather*}
A(x;q) := \sum_{n\ge0} a_n(q)x^n,
\end{gather*}
is absolutely convergent for $|x|<1$. 
As above, we may show that for $n\geq 1$, $\beta_{n+1}(q) - \beta_{n}(q) = q^{n}a_{n}(q)$. 
Proceeding as in \eqref{btail} yields that
${\displaystyle\lim_{n\rightarrow\infty}}\left( \sum_{1\le \ell\le n}f(\ell) - a_n(q) \right)$
exists.

Again we find that by Proposition \ref{P:Appell} that
\begin{gather*}
\lim_{n\rightarrow\infty}\left( \sum_{1\le \ell\le n}f(\ell) - a_n(q) \right)
=
\lim_{x\rightarrow1^-}(1-x)B(x;q)
=
A(q;q).
\end{gather*}

To compute $A(q;q)$, we set
\begin{gather*}
F(x) := \sum_{n\ge1} f(n) x^n,
\end{gather*}
to obtain
\begin{gather*}
A(q;q)
=
	\sum_{n\ge0} \frac{ (-1)^n F\left(q^{n+1}\right) q^{\frac{n(n+1)}{2}}}{(q;q)_{n+1}}
.
\end{gather*}
The calculations now differ from that of the proof of Theorem \ref{Theorem:GeneralizationOne}. Since
\begin{gather*}
\sum_{0\le k\le N-1} \zeta_N^{(n-j)k}
=
\begin{cases}
	N & \mbox{if } n\equiv j \pmod{N}
	,\\
	0 & \mbox{otherwise},
\end{cases}
\end{gather*}
we have
\begin{gather}\label{formulafn}
f(n)
=
\sum_{1\le j\le N} \frac{f(j)}{N}
\sum_{0\le k\le N-1} \zeta_N^{(n-j)k}.
\end{gather}
Thus, for $|x|<1$, we obtain
\begin{align*}
F(x)
&=
\sum_{0\le k\le N-1}
\frac{c_k  x}{1-\zeta_N^k x}
.
\end{align*}

To complete the proof, we compute, summing the $_2\phi_1$ terms with Lemmas \ref{lem:S0} and \ref{lem:Cauchy} 
\begin{align*}
A(q;q)
&=
	-c_0\sum_{n\ge1} \frac{(-1)^n q^{\frac{n(n+1)}{2}}}{(1-q^n)(q;q)_n}
	-
	\sum_{1\le k\le N-1} c_k
	\sum_{n\ge1}
	\frac{(-1)^n q^{\frac{n(n+1)}{2}}}{(1-\zeta_N^kq^n)(q;q)_n}
\\
&=
	c_0S_0(q)
	+
	\sum_{1\le k\le N-1} \frac{c_k}{1-\zeta_N^k}
	-
	(q;q)_\infty
	\sum_{1\le k\le N-1} 
	\frac{c_k}{(\zeta_N^k;q)_\infty}.
\end{align*}
\end{proof}
\vspace{-2em}\qed
\vspace{2em}

With the proof of Theorem \ref{Theorem:GeneralizationTwo} finished, we can
now easily prove Corollary \ref{cor:recur}.

\begin{proof}[Proof of Corollary \ref{cor:recur}]
To prove the statement of the corollary, we rewrite the terms in Theorem \ref{Theorem:GeneralizationTwo}. 
Firstly, Lemma \ref{lem:rootid} gives that
\begin{align*}
\sum_{1\le k\le N-1}\frac{c_k}{1-\zeta_N^k}
&=
	\frac{1}{N}
	\sum_{1\le j\le N}
	f(j)
	\left( \frac{N+1}{2}-j \right).
\end{align*}
Similarly, we obtain, using Lemma \ref{lem:xqsum}, Lemma \ref{lem:S0}, and Lemma \ref{lem:rootid},
\begin{multline*}
\sum_{1\le k\le N-1}\frac{c_k}{(\zeta_N^k;q)_\infty}
=
	\frac{1}{N(q;q)_\infty}
	\sum_{1\le j\le N} f(j)
	\left( \frac{N+1}{2} - j \right)
	+
	\frac{S_0(q)}{N(q;q)_\infty}
	\sum_{1\le j\le N} f(j)
	\\\quad
	-
	\sum_{n\ge0}\frac{q^n}{(q;q)_n}
	\sum_{1\le j\le N} f(j) \left\lceil\frac{n+1-j}{N}\right\rceil
.
\end{multline*}
Plugging these values into the statement of Theorem 
\ref{Theorem:GeneralizationTwo} yields the claim.
\end{proof}

\section{Two Examples}

Choosing $f(n) = (-1)^n$, Theorem \ref{Theorem:GeneralizationTwo} and Corollary \ref{cor:recur} directly imply the following.
\begin{corollary}\label{cor:imply}
Suppose that $a_n(q)$ is defined recursively by $a_n(q)=(-1)^n+(1-q^{n-1})a_{n-1}(q)$ for $n\in\mathbb{N}$ and $a_0(q)=0$.
Then we have
\begin{equation*}
\lim_{n\to \infty} \left(\sum_{1\leq \ell\leq n} (-1)^\ell - a_n(q)\right) 
= 
	-\frac{1}{2}-\frac{(q;q)_\infty}{2(-q;q)_\infty}
=
	\sum_{n\geq 1} (-1)^n q^{n^2} 
= 
	-(q;q)_\infty \sum_{n\geq 0} \frac{q^{2n+1}}{(q;q)_{2n+1}}.
\end{equation*}
\end{corollary}

Furthermore, letting $q\mapsto q^2$ and setting $b=q$ in Theorem \ref{Theorem:GeneralizationOne}
gives the following related identity.
\begin{corollary}\label{cor:recur2}
Suppose that $a_n(q^2)$ is defined recursively by
$a_n(q^2) = q^n +(1-q^{2n-2})a_{n-1}(q^2)$ for $n\in\N$ and $a_0(q^2) = 0$.
Then we have
\begin{equation*}
q^{-1}
\lim_{n\to \infty} a_n\left(q^2\right) 
= 
	\frac{\left(q^2;q^2\right)_\infty}{\left(q;q^2\right)_\infty}
	=
	\sum_{n\geq 0} q^{\frac{n(n+1)}{2}} 
.
\end{equation*}
\end{corollary}

\end{document}